\documentclass[12pt,twoside]{article}

\usepackage{amsmath}
\usepackage{amssymb}
\usepackage{amsfonts}
\usepackage{amsthm}
\usepackage{fancyhdr}

\usepackage[a4paper,inner=25mm,outer=25mm,top=30mm,bottom=30mm,marginparsep=5mm,marginparwidth=40mm]{geometry} 

\theoremstyle{plain}
\newtheorem{theorem}{Theorem}[section]
\newtheorem{lemma}[theorem]{Lemma}
\newtheorem{corollary}[theorem]{Corollary}
\newtheorem{proposition}[theorem]{Proposition}
\numberwithin{equation}{section}

\theoremstyle{definition}
\newtheorem{remark}[theorem]{Remark}
\newtheorem{definition}[theorem]{Definition}

\numberwithin{equation}{section}

  \newcommand{\N}{\mathbb{N}}

  \newcommand{\C}{\mathbb{C}}
  \newcommand{\om}{\Omega}

  \newcommand{\ch}{\mathcal{H}}

  \newcommand{\inn}{\emph{int\,}}
  \newcommand{\eps}{\varepsilon}

\newcommand{\mc}{M\raise.5ex\hbox{\small{c}}}  %to put the little c in the right place in John's name

\begin{document}
\title{Inverse and Implicit Function Theorems for Noncommutative Functions on Operator Domains}
\author{Mark E. Mancuso}
\newcommand{\address}{
 Mark E. Mancuso, Department of Mathematics and Statistics, Washington University in St. Louis,
    St. Louis, MO 63130

  \emph{E-mail address}:  \texttt{mark.mancuso@wustl.edu}}
\date{}
\pagestyle{fancy}
\fancyhead[R]{}
\fancyhead[L]{}
\fancyhead[LE]{\thepage}
\fancyhead[RO]{\thepage}
\fancyhead[CO]{Operator NC Inverse and Implicit Function Theorems}
\fancyhead[CE]{Mark E. Mancuso}
\fancyfoot{}
\setlength{\headheight}{15pt}
\renewcommand{\headrulewidth}{0pt}

\maketitle

\begin{abstract}
Classically, a noncommutative function is defined on a graded domain of tuples of square matrices. In this note, we introduce a notion of a noncommutative function defined on a domain $\om \subset B(\ch)^d$, where $\ch$ is an infinite dimensional Hilbert space. Inverse and implicit function theorems in this setting are established. When these operatorial noncommutative functions are suitably continuous in the strong operator topology, a noncommutative dilation-theoretic construction is used to show that the assumptions on their derivatives may be relaxed from boundedness below to injectivity.
\end{abstract}


\bigskip


\section*{INTRODUCTION} \label{sec:introduction}

  Polynomials in $d$ noncommuting indeterminates can naturally be evaluated on $d$-tuples of square matrices of any size. The resulting function is graded (tuples of $n\times n$ matrices are mapped to $n\times n$ matrices) and preserves direct sums and similarities. Along with polynomials, noncommutative rational functions and  power series, the convergence of which has been studied for example in \cite{book}, \cite{pop1}, \cite{pop2}, serve as prototypical examples of a more general class of functions called \emph{noncommutative functions}. The theory of noncommutative functions finds its origin in the 1973 work of J. L. Taylor \cite{jtaylor}, who studied the functional calculus of noncommuting operators. Roughly speaking, noncommutative functions are to polynomials in noncommuting variables as holomorphic functions from complex analysis are to polynomials in commuting variables.

  Noncommutative functions are classically defined on domains sitting inside of a graded space of $d$-tuples of square matrices which is closed under direct sums. These matrices are usually over the complex numbers, but much of the theory works for matrices over a general module over a commutative ring. See the book by D. S. Kaliuzhnyi-Verbovetskyi and V. Vinnikov \cite{book} for a comprehensive, foundational treatment in this generality. In the complex case, for example, this means that a (matricial) noncommutative function is defined on a domain $D\subset M^d:=\bigsqcup_{n=1}^{\infty} M_n^d,$ where $M_n$ is the space of $n \times n$ complex matrices, and $D$ is assumed to be open in the Euclidean topology at each level and closed under direct sums: $x\in D$ at level $n$ and $y\in D$ at level $m$ implies $x\oplus y \in D$ at level $n+m$. A noncommutative function on $D$ is a graded function $f:D \rightarrow M^r$ which preserves direct sums and similarities: $x,y \in D$ and $s$ invertible with $s^{-1}xs \in D$ implies $f(x\oplus y)=f(x)\oplus f(y)$ and $f(s^{-1}xs)=s^{-1}f(x)s.$

  In this note, we consider noncommutative functions defined on a domain $\om \subset B(\ch)^d$ for an infinite dimensional separable Hilbert space $\ch.$ Noncommutative polynomials, rational functions, and power series may again be naturally evaluated at operator tuples in a suitable domain inside of $B(\ch)^d.$ With this point of view of a noncommutative function, we are no longer considering a space of infinitely many disjoint levels, but instead are working with a \emph{complete} space. This should be seen as a type of \emph{completion} of the classical matricial noncommutative setting. In this operatorial setting, noncommuative functions are still defined to be direct sum-preserving, but since the domain is no longer graded, we need to make identifications of $\ch$ with countable direct sums of $\ch$ via unitary equivalence. The precise definitions and further discussion will be given in Section \ref{sec: prelim}. Many foundational properties and formulas from the matricial theory, such as those found in the work of Helton, Klep, and McCullough \cite{HKM}, have analogues in this setting. We give their formulations and proofs in Section \ref{sec: foundational}. For example, a standard derivative formula now takes the form \[f\left(s^{-1}\begin{bmatrix}
   x & h \\
   0 & x
   \end{bmatrix}s\right)=s^{-1}\begin{bmatrix}
   f(x) & Df(x)[h] \\
   0 & f(x)
   \end{bmatrix}s,\]   where $s:\ch \rightarrow \ch \oplus\ch$ is linear and invertible. In the interest of clarity, when dealing with noncommuative functions on operator domains inside of $B(\ch)^d,$ we will use the abbreviation NC, and use the lower case nc for the matricial noncommutative setting. Agler and \mc Carthy proposed a definition similar to ours for NC functions on operator domains in \cite{amop} and gave a set of equivalent conditions for when such functions are approximable by  polynomials and have a realization formula on a polynomial polyhedron.

   Our main results are (global) inverse and implicit function theorems in the operatorial noncommutative setting. It is here that completeness and the structure and topologies of $B(\ch)$ play a key role. The inverse function theorem of J. E. Pascoe \cite{james} gave necessary and sufficient conditions for a matricial nc function to be invertible in terms of injectivity of its derivative map at all points. We prove similar results for operator NC functions.  In \cite{akv15}, quite general matricial nc results on the inverse and implicit function theorems are obtained in the setting of operator spaces and nilpotent matrices. In that paper, the authors exploit the existence of a natural "uniformly-open" topology and consider nc functions that are locally bounded in this topology which also have a completely bounded and invertible derivative. As we are not working with functions on graded domains in this note, such a topology is unavailable to us in our study of operator NC functions. In further contrast to the work in \cite{akv15} and other articles on noncommutative inversion, we give a sufficient condition guaranteeing the invertibility of the derivative map of an NC function at all points in a connected domain. Indeed, Theorem \ref{bbthm} states that for an NC function $f$ on a connected domain in $B(\ch)^d$, if the derivative $Df$ satisfies a noncommutative bounded below condition (see Definition \ref{ncbbp}) and we assume the existence of just one point $a$ in the domain such that $Df(a)$ is invertible, then we may conclude the invertibility of $Df(x)$ for \emph{every} $x$ in the domain. This result provides the basis for the inverse and implicit function theorems \ref{inverse} and \ref{implicit}.

   Finally, we end this note by considering operator NC functions that are continuous (in a precise sense detailed in Section \ref{sec: strongshift}) in the \emph{strong operator topology}. In fact, this allows us to further weaken the assumptions on the derivative maps. There does not appear to be much in the literature on the connection between noncommutative inversion results and continuity in the strong operator topology, but it is reasonable to impose this extra continuity condition on NC functions since the examples of interest in most applications (such as polynomials, rational functions, etc.) are strongly continuous on appropriately defined norm-bounded sets. In Section \ref{sec: strongshift}, ideas from noncommutative dilation theory are used to prove certain convergence and compactness-like results in the strong operator topology that interact well with noncommutative function theory. It is proved as a consequence of these results, in Theorem \ref{strongbb} and Corollary \ref{strongbbcor}, that injective strongly continuous NC functions, on suitable domains, have everywhere bounded below derivative. Therefore, in the operator setting, and especially in the case of strong operator continuity, we are able to obtain global inversion-type theorems with minor hypotheses on the derivative. 

  \section{PRELIMINARIES} \label{sec: prelim}

  In this section, we elaborate on our general setting and provide definitions and examples of our main objects of study: NC operator domains and functions. Operator noncommutative functions are to be defined on domains sitting inside of $B(\ch)^d$, where $\ch$ is an infinite dimensional separable Hilbert space over $\C$ and $B(\ch)$ is the Banach space of bounded linear operators on $\ch$ equipped with the operator norm.  Elements of $B(\ch)^d$ will sometimes be written as $d$-tuples with superscripts such as $(x^1,\ldots, x^d).$ We equip $B(\ch)^d$ with the maximum norm $$\Vert x\Vert := \max\{\Vert x^1\Vert, \ldots, \Vert x^d\Vert \},$$ which induces the product topology on $B(\ch)^d$ with respect to the norm topology on $B(\ch)$ and turns $B(\ch)^d$ into a complex Banach space. 

  The direct sum of $l$ copies of the Hilbert space $\ch,$ for $l\in \N \cup \{\infty\},$ will be denoted $\ch^{(l)}$.
  Direct sums of operators will often be written as a diagonal matrix: if $x_1, x_2 \ldots$ is a finite or countably infinite sequence of operators in $B(\ch)$ of length $l\in \N \cup \{\infty\},$ we will write the direct sum operator $\bigoplus_{i=1}^l x_i$ as the diagonal matrix $$\begin{bmatrix}
    x_{1} & &  \\
    & x_{2} & \\
    & & \ddots     
  \end{bmatrix}: \ch^{(l)} \rightarrow \ch^{(l)}.$$

  Operations on $B(\ch)^d$ are defined component-wise: for $L\in B(\ch)$ and $x\in B(\ch)^d,$ define $$L(x^1,\ldots, x^d):=(Lx^1,\ldots, Lx^d) \,\,\,\, \text{and} \,\,\,\, (x^1,\ldots, x^d)L:=(x^1L,\ldots, x^dL).$$ Similarly, if $s: \ch \rightarrow \ch^{(l)}$ is an invertible linear map and $z\in B(\ch^{(l)})^d$, we define $$s^{-1}zs:=(s^{-1}z^1s,\ldots,s^{-1}z^ds).$$  Direct sums of operator tuples are also defined component-wise. If $x_1, x_2, \ldots$ is a finite or countably infinite sequence of elements of $B(\ch)^d$ of length $l\in \N \cup \{\infty\},$ we define their direct sum to be the element of $B(\ch^{(l)})^d$ given by $$\begin{bmatrix}
    x_{1} & & \\
    & x_{2} &  \\
    & & \ddots 
  \end{bmatrix}:= \left(\begin{bmatrix}
    x_{1}^1 & & \\
    & x_{2}^1 &  \\
    & & \ddots  
  \end{bmatrix},\ldots, \begin{bmatrix}
    x_{1}^d & &  \\
    & x_{2}^d &  \\
    & & \ddots 
  \end{bmatrix}\right).$$ Expressions such as $$\begin{bmatrix}
  x & y \\
  z & w 
\end{bmatrix}$$ for $x,y,z,w \in B(\ch)^d$ are similarly defined.
  We say a subset $\om$ of $B(\ch)^d$ is \emph{unitarily invariant} if whenever $x\in \om$ and $u \in B(\ch)$ is a unitary operator, then $u^*xu \in \om.$ In what follows, the interior of a set is with respect to the norm topology on $B(\ch)^d$.

  \begin{definition} \label{ncd}
  A set $\om\subset B(\ch)^d$ is called an \emph{NC domain} if there exists a sequence $\{\om_k\}_{k=1}^{\infty}$ of subsets of $\om$ with the following properties:

\begin{enumerate}
  \item  $\om_k \subset \emph{\inn}\om_{k+1}$ for all $k$ and $\om=\bigcup_{k=1}^{\infty} \om_k$.

  \item Each $\om_k$ is bounded and unitarily invariant.

  \item Each $\om_k$ is closed under countable direct sums: If $x_n$ is a sequence in $\om_k$ of length $l\in \N\cup \{\infty\}$, then there exists a unitary $u:\ch\rightarrow \ch^{(l)}$ such that \begin{align} \label{exhaustds} u^{-1}\begin{bmatrix}
    x_{1} & &  \\
    & x_{2} & \\
    & & \ddots  
  \end{bmatrix}u \in \om_k.\end{align}
\end{enumerate}
  \end{definition}

  NC domains are open subsets in the norm topology of $B(\ch)^d.$ Note that by unitary invariance of each level $\om_k,$ given a finite or countably infinite sequence $x_n$ in $\om_k$ of length $l,$ as soon as (\ref{exhaustds}) holds for some unitary $u:\ch\rightarrow \ch^{(l)},$ it will in fact hold for \emph{all} unitaries $v:\ch\rightarrow \ch^{(l)}$ by considering $u^{-1}v$.

  A large supply of examples of operator NC domains can be given as follows. Let $\delta$ be an $I\times J$ matrix of polynomials in $d$ noncommuting variables (i.e. a matrix whose entries are elements of the free associative algebra $\C\langle x^1,\ldots, x^d \rangle$). Define $$B_{\delta}:=\{x\in B(\ch)^d : \Vert \delta(x)\Vert<1 \},$$ where the norm is taken in $B(\ch^{(J)},\ch^{(I)}).$ Important concrete examples take this form for particular choices of $\delta$. For example, the noncommutative polydisk $\{x\in B(\ch)^d: \Vert x\Vert <1\}$ in $B(\ch)^d$ may be realized as a $B_{\delta}$ for the diagonal matrix $$\delta(x^1,\ldots, x^d)=
  \begin{bmatrix}
x^1 & & \\
& \ddots & \\
& & x^d
  \end{bmatrix}.$$ The noncommutative operatorial ball $$\{x\in B(\ch)^d : \Vert x^1(x^1)^* +\dots + x^d(x^d)^* \Vert^{1/2} <1 \}$$ is a $B_{\delta}$ for the row matrix $\delta(x)=[x^1 \cdots \, x^d]$.

  To see that any $B_{\delta}$ is in fact an NC domain according to Definition \ref{ncd}, one may take the exhausting sequence to be \begin{align}\label{bdelta} \om_k=\{x\in B(\ch)^d : \Vert \delta(x)\Vert\leq 1-1/k\}\cap \{x \in B(\ch)^d : \Vert x\Vert \leq k\}.\end{align} It is immediately checked that $\{\om_k\}$ has all of the required properties. Another example of an NC domain is the set of invertible elements of $B(\ch),$ where one may use the exhausting sequence $\om_k=\{x \in B(\ch): \Vert x\Vert \leq k, \Vert x^{-1}\Vert \leq k\}.$

  Let us make a few remarks about Definition \ref{ncd}. Our notion of operator NC domain using exhausting sequences is a way to reasonably think of the (open) domains as being closed under countably infinite direct sums while still providing a sufficiently large class of examples. Even for bounded domains, it will rarely be the case that one may take an arbitrary sequence in the domain and conclude that its direct sum (conjugated by a sufficient unitary) will remain in the domain.  Indeed, this fails even for the open unit ball in $d=1$: consider the sequence $(1-1/n)1_{\ch}$. When we restrict to sequences contained in a fixed level of an exhaustion as in Definition \ref{ncd}, however, it is a much less stringent requirement for (\ref{exhaustds}) to hold.

  We want to consider functions which act appropriately on NC domains. Namely, we make the following definition of an operator NC function.

  \begin{definition} \label{ncf}
  Let $\om \subset B(\ch)^d$ be an NC domain. We say a function $f:\om \rightarrow B(\ch)^r$ is an \emph{NC function} if it preserves direct sums in the sense that whenever $x,y \in \om$ and whenever $s:\ch\rightarrow \ch^{(2)}$ is a bounded invertible linear map with $$s^{-1}\begin{bmatrix}
    x & 0 \\
   0 & y 
  \end{bmatrix}s \in \om,$$ then $$f\left(s^{-1}\begin{bmatrix}
    x & 0  \\
   0 & y   
  \end{bmatrix}s\right)=s^{-1}\begin{bmatrix}
    f(x) & 0  \\
   0 & f(y)  
  \end{bmatrix}s.$$
  \end{definition}

  As a consequence of Lemma \ref{L}, proved in Section \ref{sec: foundational}, NC functions on operator domains preserve intertwinings, just as in the matricial nc theory: if $L\in B(\ch)$ and $Lx=yL,$ then $Lf(x)=f(y)L.$ From this observation, it follows that whenever $f:\om \subset B(\ch)^d \rightarrow B(\ch)^r$ is NC and $\{\om_k\}$ is an exhausting sequence of $\om$ as in Definition \ref{ncd}, then each $f(\om_k)$ is norm-bounded. In particular, NC functions are automatically locally bounded. To see this, take a sequence $x_n$ in a fixed $\om_k$. There is a unitary $u:\ch \rightarrow \ch^{(\infty)}$ such that $x:=u^{-1}[\bigoplus x_n ] u \in \om_k.$ Define $\Gamma_n :\ch^{(\infty)}\rightarrow \ch$ to be projection onto the $n$th component and let $L_n:=\Gamma_nu.$ By definition of $L_n,$ we have $L_nx=x_nL_n$ for all $n,$ and since $f$ preserves intertwinings, we then have $L_nf(x)=f(x_n)L_n.$ Since $f(x)$ is an element of $B(\ch)^r,$ it has finite norm, so this relation implies the $f(x_n)$ are uniformly bounded. Similar reasoning lets us conclude that operator NC functions actually preserve \emph{countable} direct sums: if $x_n$ is a sequence in $\om$ of length $l \in \N \cup \{\infty\}$ and $s:\ch \rightarrow \ch^{(l)}$ is linear and invertible with $$s^{-1}\begin{bmatrix}
    x_{1} & &  \\
    & x_{2} & \\
    & & \ddots  
  \end{bmatrix}s \in \om,$$ then $f(x_n)$ is uniformly bounded and we have $$f\left(s^{-1}\begin{bmatrix}
    x_{1} & &  \\
    & x_{2} & \\
    & & \ddots  
  \end{bmatrix}s\right)=s^{-1}\begin{bmatrix}
    f(x_{1}) & &  \\
    & f(x_{2}) & \\
    & & \ddots  
  \end{bmatrix}s.$$

   If we write $f:\om\rightarrow B(\ch)^r$ as $f=(f^1,\ldots, f^r)$, where each $f^j: \om \rightarrow B(\ch)$, then it follows from the definitions that $f$ is an NC function if and only if each $f^j$ is an NC function. As discussed in the introduction, any polynomial in $d$ noncommuting variables is an NC function when defined on any NC domain $\om \subset B(\ch)^d$. Furthermore, rational functions and noncommutative power series, on appropriately defined NC domains, provide us with a sizable class of prototypical examples of NC functions.

  For a simple, explicit such example, consider the rational function \begin{align} \label{example}f(x,y):=(1-xy)^{-1}=\sum_{n=0}^{\infty} (xy)^n\end{align} defined on the unit bidisk $\om=\{(x,y): \Vert x\Vert<1, \Vert y\Vert<1\}$ in $B(\ch)^2$. We verify here through a direct calculation that this function is in fact NC. Let $(x^1,x^2)$ and $(y^1,y^2)$ be points in $\om$ and suppose $$s^{-1}\begin{bmatrix}
(x^1,x^2) & 0 \\
0 & (y^1,y^2)  
  \end{bmatrix}s \in \om.$$ Then, \begin{align*}
f\left(s^{-1}\begin{bmatrix}
(x^1,x^2) & 0 \\
0 & (y^1,y^2)  
  \end{bmatrix}s\right)&= \sum_{n=0}^{\infty} \left(s^{-1}\begin{bmatrix}
x^1 & 0 \\
0 & y^1  
  \end{bmatrix}\begin{bmatrix}
x^2 & 0\\
0 & y^2 
  \end{bmatrix}s\right)^n \\
  &= s^{-1}\sum_{n=0}^{\infty}\begin{bmatrix}
(x^1x^2)^n & 0 \\
0 & (y^1y^2)^n 
  \end{bmatrix}s \\
  &= s^{-1}\begin{bmatrix}
f(x^1,x^2) & 0 \\
0 & f(y^1,y^2) 
  \end{bmatrix}s, 
  \end{align*} as claimed.

  We conclude this section with terminology that will be used in the statements of the inverse and implicit function theorems. Recall that an operator $T\in B(X)$, where $X$ is a Banach space, is bounded below if there is a constant $C>0$ such that $\Vert Tx\Vert\geq C\Vert x\Vert$ for all $x \in X.$

  \begin{definition} \label{ncbbp}
Let $\om \subset B(\ch)^d$ be an NC domain. A map $\Psi: \om \rightarrow B(B(\ch)^r)$ is said to have the \emph{NC bounded below property} if whenever $\{x_n\}_{n=1}^{\infty}$ is a bounded sequence in $\om$ such that $\Psi(x_n)$ is bounded below for every $n,$ and whenever $u:\ch \rightarrow \ch^{(\infty)}$ is unitary such that $$z:=u^{-1}\begin{bmatrix}
    x_{1} & &  \\
    & x_{2} & \\
    & & \ddots  
  \end{bmatrix}u \in \om,$$ then $\Psi(z)$ is bounded below.
  \end{definition}

  We note that in the notation of Definition \ref{ncbbp}, the supposed bound below for $\Psi(x_n)$ is allowed to depend on $n.$ When $\Psi$ arises naturally from an operator NC function, for example when $\Psi$ is the derivative map of an NC function, the argument in the proof of Theorem \ref{bbthm} shows that $\Psi$ being bounded below when evaluated at the direct sum of such a sequence $x_n$ implies a \emph{uniform} bound below for the sequence $\Psi(x_n).$ As a result, we show that this property characterizes global invertibility of the derivative of NC functions on connected NC domains. The NC bounded below property, when imposed on the derivative map, may be thought of as an operatorial analogue of injectivity of the derivative.

    \section{MAIN RESULTS} \label{sec:main}

  In this section, we list the main results of the paper; for the detailed proofs, see Section \ref{sec: proofs}.  The notion of connectedness is always with respect to the norm topology on $B(\ch)^d$.  Operator NC functions will be shown to automatically be Fr\'{e}chet differentiable, and the notation $Df(x)$ denotes the derivative mapping $B(\ch)^d \rightarrow B(\ch)^r$ of $f$ at the point $x$ in the domain of $f$. We denote by $Df$ the map $x\mapsto Df(x).$

  Our primary objective is to prove an inverse function theorem for NC functions defined on operator domains, as described in Section \ref{sec: prelim}. Therefore, we begin by studying the derivative of such functions and ask when are the derivative maps $Df(x)$ invertible for \emph{every} $x$ in the domain of $f.$ Theorem \ref{bbthm} below provides an answer to this question on connected NC domains.

\begin{theorem} \label{bbthm}
Let $f:\om \subset B(\ch)^d \rightarrow B(\ch)^d$ be an NC function and suppose $\om$ is connected. If $Df$ has the NC bounded below property, and there exists a point $a\in \om$ such that $Df(a)$ is invertible, then $Df(x)$ is invertible for every $x\in \om.$
\end{theorem}

With this result giving a sufficient condition for the invertibility of the derivative map of an NC function at \emph{all} points, we arrive at an operatorial NC inverse function theorem. Theorem \ref{bbthm} justifies the NC bounded below property as a substitute for injectivity in the general operatorial setting. The hypotheses for the inverse function theorem are the same as in Theorem \ref{bbthm}. 

\begin{theorem} \label{inverse} \emph{(Inverse Function Theorem)}
Let $f:\om \subset B(\ch)^d \rightarrow B(\ch)^d$ be an NC function and suppose $\om$ is connected. If $Df$ has the NC bounded below property, and there exists a point $a\in \om$ such that $Df(a)$ is invertible, then $f(\om)$ is an NC domain and $f^{-1}: f(\om) \rightarrow \om$ exists and is an NC function.
\end{theorem}

As one might expect, the inverse function theorem, Theorem \ref{inverse}, gives rise to an operatorial implicit function theorem under the hypothesis that an augmented derivative map satisfies the NC bounded below property. The notation $Z_f$ denotes the zero set of the function $f$. In the implicit function theorem, we write, for notational convenience, $(h^{d-r+1},\ldots, h^d)$ for elements of $B(\ch)^r.$ 

\begin{theorem} \label{implicit} \emph{(Implicit Function Theorem)}
Let $f:\om \subset B(\ch)^d \rightarrow B(\ch)^r$ be NC, where $1\leq r\leq d-1$, and $\om$ is connected. Suppose the map $\Psi:\om \rightarrow B(B(\ch)^r)$ defined by $$\Psi(x)(h^{d-r+1},\ldots, h^d)=Df(x)[0,\ldots,0,h^{d-r+1},\ldots, h^d]$$ has the NC bounded below property, and there exists a point $a \in \om$ such that $\Psi(a)$ is invertible.

Then, there exists $V\subset B(\ch)^{d-r}$ an NC domain and $\phi:V\rightarrow B(\ch)^r$ an NC function such that $$Z_f=\{(y,\phi(y)): y \in V\}.$$ Furthermore, $V$ is given by the projection onto the first $d-r$ coordinates of the zero set $Z_f.$
\end{theorem}

Theorem \ref{implicit} is an operatorial analogue of Agler and \mc Carthy's implicit function theorem (Theorem 6.1 in \cite{amif16}) for the fine matricial nc topology. In the operatorial setting, we require a slightly stronger assumption than merely injectivity of the maps $\Psi(x)$, which is the assumption for the implicit function theorem in \cite{amif16}. For further emphasis, the parametrizing function $\phi$ in Theorem \ref{implicit}, being itself operator NC, is \emph{infinite} direct sum-preserving. It is important to note that the conclusions of Theorems \ref{inverse} and \ref{implicit} are \emph{global}, a phenomenon that is rare outside of the noncommutative setting. As mentioned in the introduction, results similar to Theorems \ref{inverse} and \ref{implicit} are obtained in \cite{akv15} for a quite general matricial nc setting with local invertibility conclusions and hypotheses of analyticity in the "uniformly-open" topology and a completely bounded and invertible derivative map with completely bounded inverse. In the operator NC setting, we note once more that the notion of a uniformly-open topology is no longer available, so we instead make extensive use of the completeness of $B(\ch)$ and its various topologies.

It is reasonable to ask if additional structure imposed on the NC functions in the strong operator topology (SOT) allows us to weaken our assumptions on their derivatives. If we assume the NC operator domain is exhausted by certain SOT-closed sets, and impose strong continuity on the NC function, we arrive at the following, rather surprising theorem. In particular, it is valid for maps whose components are polynomials and rational functions, as these are SOT continuous on appropriate norm-bounded sets. 

\begin{theorem} \label{strongbb}
Suppose $f:\om \subset B(\ch)^d \rightarrow B(\ch)^r$ is a strong NC function. If $Df(x)$ is injective for every $x\in \om$, then $Df(x)$ is bounded below for every $x\in \om.$
\end{theorem}

We note that Theorem \ref{strongbb} does not require the NC domain to be connected in any topology. On the other hand, injective strong NC functions on norm-connected domains are especially nice:

\begin{corollary} \label{strongbbcor}
Let $f:\om \subset B(\ch)^d \rightarrow B(\ch)^d$ be an injective strong NC function. If $\om$ is connected and there exists a point $a\in \om$ such that $Df(a)$ is surjective, then $Df(x)$ is invertible for every $x\in \om.$
\end{corollary}

Results such as Theorem \ref{strongbb} and Corollary \ref{strongbbcor} suggest it may be natural to have some structure in the strong operator topology  built into the definitions of NC domain and function. However, Theorems \ref{bbthm}, \ref{inverse}, and \ref{implicit}, along with the foundations found in Section \ref{sec: foundational}, require no such hypotheses. As such, there is merit to also studying a more general theory. Therefore, we maintain a distinction throughout this note.

See Section \ref{sec: strongshift} for more details on the precise definition of \emph{strong NC function} and the construction of what we call \emph{shift forms}. Reminiscent of noncommutative dilation theory, these shift forms have nice SOT convergence properties (Lemma \ref{sotcompact}) that are suited well for applications to strong NC functions.

  \section{FOUNDATIONAL PROPERTIES}\label{sec: foundational}

  The aim of this section is to collect basic properties and formulas for NC functions defined on operator domains. Our first lemma is an operatorial version of a fundamental formula for noncommutative functions. In \cite{HKM}, Helton, Klep, and McCullough proved a similar formula for matricial nc functions. In this and other related formulas to follow, the presence of unitaries or some invertible linear map $s$ in the statements is necessary as we need a way of identifying $\ch$ with some $\ch^{(l)}.$ Several results in this section have analogues in the classical matricial nc theory. However, we present precise statements and complete proofs here, adhering to the formalisms introduced in Section \ref{sec: prelim}. 

  \begin{lemma} \label{L}
  Let $f:\om \subset B(\ch)^d\rightarrow B(\ch)^r$ be an NC function and let $L\in B(\ch)$. If $x,y \in \om$ and $s:\ch \rightarrow \ch^{(2)}$ is any invertible linear map such that $$s^{-1}\begin{bmatrix}
  x & Ly-xL \\
  0 & y{}
  \end{bmatrix}s \in \om,$$ then $$f\left(s^{-1}\begin{bmatrix}
  x & Ly-xL \\
  0 & y
  \end{bmatrix}s\right)=s^{-1}\begin{bmatrix}
  f(x) & Lf(y)-f(x)L \\
  0 & f(y)
  \end{bmatrix}s.$$
  \end{lemma}
  \begin{proof}
    Define $\sigma:\ch\rightarrow \ch^{(2)}$ to be the invertible map $\sigma:=\begin{bmatrix}
  1 & -L \\
  0 & 1 
  \end{bmatrix}s.$ Then a computation shows $$\sigma^{-1} \begin{bmatrix}
  x & 0 \\
  0 & y 
  \end{bmatrix}\sigma=s^{-1}\begin{bmatrix}
  1 & L \\
  0 & 1 
  \end{bmatrix}\begin{bmatrix}
  x & 0 \\
  0 & y 
  \end{bmatrix}\begin{bmatrix}
  1 & -L \\
  0 & 1 
  \end{bmatrix}s = s^{-1}\begin{bmatrix}
  x & Ly-xL \\
  0 & y
  \end{bmatrix}s \in \om.$$ Since $f$ is NC, we have \begin{align*}
f\left(s^{-1}\begin{bmatrix}
  x & Ly-xL \\
  0 & y
  \end{bmatrix}s\right) &= f\left(\sigma^{-1}\begin{bmatrix}
  x & 0 \\
  0 & y 
  \end{bmatrix}\sigma\right) \\
  &= \sigma^{-1}\begin{bmatrix}
  f(x) & 0 \\
  0 & f(y) 
  \end{bmatrix}\sigma \\
  &= s^{-1}\begin{bmatrix}
  f(x) & Lf(y)-f(x)L \\
  0 & f(y)
  \end{bmatrix}s,
  \end{align*} which completes the proof.
  \end{proof}

  As noted previously, it immediately follows from Lemma \ref{L} that operator NC functions preserve intertwinings.

  Recall that if $X$ and $Y$ are Banach spaces and $U\subset X$ is open, then a function $g:U\rightarrow Y$ is said to be \emph{G\^{a}teaux differentiable} if for all $x\in U$ and all $h\in X$, the limit $$Dg(x)[h]:=\lim_{t\rightarrow 0} \frac{g(x+th)-g(x)}{t}$$ exists. It is a well-known general fact (see \cite{taylor}) that over complex scalars, a norm-continuous and G\^{a}teaux differentiable function is automatically \emph{Fr\'{e}chet} differentiable, and the two derivatives must then coincide. In particular,  $Dg(x):X\rightarrow Y$ is then a bounded linear map for each $x\in U$. 

  \begin{lemma} \label{difflemma}
  An NC function is norm-continuous and G\^{a}teaux differentiable, and therefore is Fr\'{e}chet differentiable.
  \end{lemma}

  \begin{proof}
We begin by showing that if $f:\om \subset B(\ch)^d\rightarrow B(\ch)^r$ is NC, then $f$ is norm-continuous. Fix $x \in \om$ and $\eps>0,$ and let $\{\om_k\}$ be an exhausting sequence for $\om$ as in Definition \ref{ncd}. Say $x\in \om_k,$ so there is $u:\ch \rightarrow \ch^{(2)}$ unitary such that $$z:=u^{-1}\begin{bmatrix}
x & 0 \\
0 & x
\end{bmatrix}u\in \om_k.$$ Then there is some $r>0$ such that the balls centered at $x$ and $z$ with radius $r$ are contained in $\om_{k+1}.$ By the discussion immediately following Definition \ref{ncf}, there is $M>0$ such that $\Vert f\Vert < M$ on $\om_{k+1}$.

Now, set $\delta:=\min\{\frac{r\eps}{2M}, r/2\}$ and let $\Vert y-x\Vert <\delta.$ Then \begin{align*}
\left \Vert u^{-1}\begin{bmatrix}
x & \frac{M}{\eps}(y-x) \\
0 & y
\end{bmatrix}u - z\right\Vert &= \left\Vert \begin{bmatrix}
0 & \frac{M}{\eps}(y-x) \\
0 & y-x
\end{bmatrix}\right\Vert \\
&\leq M/\eps \Vert y-x\Vert +\Vert y-x\Vert \\
&< r,
\end{align*} so we have, by Lemma \ref{L}, $$\left\Vert \begin{bmatrix}
f(x) & \frac{M}{\eps}(f(y)-f(x)) \\
0 & f(y)
\end{bmatrix}\right\Vert= \left\Vert f\left(u^{-1}\begin{bmatrix}
x & \frac{M}{\eps}(y-x) \\
0 & y
\end{bmatrix}u\right)\right\Vert<M.$$ It then follows that $\Vert f(y)-f(x)\Vert<\eps.$

Next, we show $f$ is G\^{a}teaux differentiable. Fix $x\in \om$ and $h\in B(\ch)^d.$ There is $k\geq 1,$ $u:\ch \rightarrow \ch^{(2)}$ unitary, and $\eps>0$ small so that $x\in \om_k$ and $$u^{-1}\begin{bmatrix}
x & \eps h \\
0 & x 
\end{bmatrix}u \in \om_k.$$ Then for all $t\neq 0$ with small enough modulus, $$\om_{k+1}\ni u^{-1}\begin{bmatrix}
x+th & \eps h \\
0 & x
\end{bmatrix}u = u^{-1}\begin{bmatrix}
x+th & \frac{\eps}{t}(x+th-x) \\
0 & x
\end{bmatrix}u,$$ so by Lemma \ref{L} again, \begin{align}\label{diffquo} f\left(u^{-1}\begin{bmatrix}
x+th & \eps h \\
0 & x
\end{bmatrix}u\right)=u^{-1}\begin{bmatrix}
f(x+th) & \frac{\eps}{t}(f(x+th)-f(x)) \\
0 & f(x)
\end{bmatrix}u.\end{align} By continuity of $f$, as $t\rightarrow 0$, the limit on the left-hand side of (\ref{diffquo}) exists, and therefore so does that of the 1-2 entry of the matrix on the right-hand side of (\ref{diffquo}), thus proving $f$ is G\^{a}teaux differentiable. Since $f$ is also continuous, the discussion immediately preceding this proof implies $f$ is Fr\'{e}chet differentiable.
\end{proof}

  Moreover, the second part of the above proof also provides the following derivative formula for operator NC functions. It is reminiscent of a formula obtained in \cite{HKM}, and will be an irreplaceable tool for us moving forward. 

  \begin{proposition} \label{derprop}
   Let $f:\om \subset B(\ch)^d\rightarrow B(\ch)^r$ be an NC function. Suppose $x\in \om$, $h\in B(\ch)^d,$ and $s:\ch \rightarrow \ch^{(2)}$ is any invertible linear map such that $$s^{-1}\begin{bmatrix}
   x & h \\
   0 & x
   \end{bmatrix}s \in \om.$$ Then, \begin{align} \label{derform} f\left(s^{-1}\begin{bmatrix}
   x & h \\
   0 & x
   \end{bmatrix}s\right)=s^{-1}\begin{bmatrix}
   f(x) & Df(x)[h] \\
   0 & f(x)
   \end{bmatrix}s.\end{align}   
  \end{proposition}

  A common scenario where we can apply Proposition \ref{derprop} is as follows. Suppose $x\in \om$ and $s=u$ is a given unitary. Then by closure under direct sums and unitary invariance, $u^{-1}(x\oplus x)u$ is an element of $\om$ (for the \emph{given} unitary $u$) and the conclusion of Proposition \ref{derprop} holds for all $h \in B(\ch)^d$ with sufficiently small norm.

  The next theorem is an operatorial analogue of J. E. Pascoe's inverse function theorem \cite{james} for matricial nc functions. It is a first step towards a bonafide inverse function theorem for operator NC functions. We remark that, in contrast to the finite dimensional case, it is possible for a linear map $B(\ch)^d\rightarrow B(\ch)^r$ to be injective even if $d>r.$  Therefore, this theorem has content even when $d\neq r,$ and so we state it in this generality.

\begin{theorem}\label{injthm}
An NC function $f:\om \subset B(\ch)^d \rightarrow B(\ch)^r$ is injective if and only if $Df(x):B(\ch)^d\rightarrow B(\ch)^r$ is injective for every $x\in \om.$ 
\end{theorem}

\begin{proof}
Suppose first $f$ is injective and let $x\in \om$. Assume that $Df(x)[h]=0$. There is $u$ unitary and $\eps>0$ small enough so that $u^{-1}\begin{bmatrix}
x & \eps h \\
0 & x
\end{bmatrix}u \in \om$. Formula (\ref{derform}) then yields \begin{align*}
f\left(u^{-1}\begin{bmatrix}
x & \eps h \\
0 & x
\end{bmatrix}u\right) &= u^{-1}\begin{bmatrix}
   f(x) & Df(x)[\eps h] \\
   0 & f(x)
   \end{bmatrix}u \\
   &= f\left(u^{-1}\begin{bmatrix}
   x & 0 \\
   0 & x 
   \end{bmatrix}u\right).
\end{align*} By injectivity of $f,$ it must hold that $$u^{-1}\begin{bmatrix}
x & \eps h \\
0 & x
\end{bmatrix}u=u^{-1}\begin{bmatrix}
x & 0 \\
0 & x
\end{bmatrix}u,$$ which implies $h=0.$ Thus, $Df(x)$ has trivial kernel.

To prove the converse, suppose $x,y \in \om$ and $f(x)=f(y).$ There are unitaries $u,v: \ch \rightarrow \ch^{(2)}$ and $\eps>0$ such that $v^{-1}\begin{bmatrix}
x & 0 \\
0 & y 
\end{bmatrix}v \in \om$ and $$z:=u^{-1}\begin{bmatrix}
v^{-1}\begin{bmatrix}
x & 0 \\
0 & y 
\end{bmatrix}v & v^{-1}\begin{bmatrix}
0 & \eps(x-y) \\
0 & 0 
\end{bmatrix}v \\
0 & v^{-1}\begin{bmatrix}
x & 0 \\
0 & y 
\end{bmatrix}v
\end{bmatrix}u \in \om.$$ First, by Proposition \ref{derprop}, and because $f$ preserves direct sums, we know \begin{align}\label{bigmatrix}
f(z)= u^{-1}\begin{bmatrix}
v^{-1}\begin{bmatrix}
f(x) & 0 \\
0 & f(y) 
\end{bmatrix}v & Df\left(v^{-1}\begin{bmatrix}
x & 0 \\
0 & y 
\end{bmatrix}v\right)\left[v^{-1}\begin{bmatrix}
0 & \eps(x-y) \\
0 & 0 
\end{bmatrix}v\right] \\
0 & v^{-1}\begin{bmatrix}
f(x) & 0 \\
0 & f(y) 
\end{bmatrix}v
\end{bmatrix}u.
\end{align} 

On the other hand, a calculation shows that if we define $w:\ch \rightarrow \ch^{(4)}$ by $w:=(v\oplus v)u$ and $s:\ch \rightarrow \ch^{(4)}$ by $$s:=\begin{bmatrix}
1 & 0 & 0 & \eps 1 \\
0 & 1 & 0 & 0 \\
0 & 0 & 1 & 0 \\
0 & 0 & 0 & 1
\end{bmatrix}w,$$ then $z$ may be rewritten as $$z=s^{-1}\begin{bmatrix}
x & 0 & 0 & 0 \\
0 & y & 0 & 0 \\
0 & 0 & x & 0 \\
0 & 0 & 0 & y
\end{bmatrix}s.$$  Therefore, as $f$ is NC, we have \begin{align*}
f(z)&= f\left(s^{-1}\begin{bmatrix}
x & 0 & 0 & 0 \\
0 & y & 0 & 0 \\
0 & 0 & x & 0 \\
0 & 0 & 0 & y
\end{bmatrix}s\right) \\
&= s^{-1}\begin{bmatrix}
f(x) & 0 & 0 & 0 \\
0 & f(y) & 0 & 0 \\
0 & 0 & f(x) & 0 \\
0 & 0 & 0 & f(y)
\end{bmatrix}s \\
&= w^{-1}\begin{bmatrix}
f(x) & 0 & 0 & \eps(f(x)-f(y)) \\
0 & f(y) & 0 & 0 \\
0 & 0 & f(x) & 0 \\
0 & 0 & 0 & f(y)
\end{bmatrix}w \\
&= w^{-1}\begin{bmatrix}
f(x) & 0 & 0 & 0 \\
0 & f(y) & 0 & 0 \\
0 & 0 & f(x) & 0 \\
0 & 0 & 0 & f(y)
\end{bmatrix}w \\
&= u^{-1} \begin{bmatrix}
v^{-1}\begin{bmatrix}
f(x) & 0 \\
0 & f(y)
\end{bmatrix}v & 0 \\
0 & v^{-1}\begin{bmatrix}
f(x) & 0 \\
0 & f(y)
\end{bmatrix}v
\end{bmatrix}u.
\end{align*} Comparing this to equation (\ref{bigmatrix}) implies \[Df\left(v^{-1}\begin{bmatrix}
x & 0 \\
0 & y 
\end{bmatrix}v\right)\left[v^{-1}\begin{bmatrix}
0 & \eps(x-y) \\
0 & 0 
\end{bmatrix}v\right]=0\] in $B(\ch)^r.$ By the assumption of the derivative being injective at all points, \[v^{-1}\begin{bmatrix}
0 & \eps(x-y) \\
0 & 0 
\end{bmatrix}v=0,\] and we conclude $x=y$ as desired.
\end{proof}

Other results on this type of  "lack of dimensionality" were observed by Cushing, Pascoe, and Tully-Doyle in \cite{cushingpascoe}.  Theorem \ref{injthm} already provides a stark contrast between classical function theory and the noncommutative theory; examples abound of functions with globally invertible derivative who fail to be injective. 

We now recall the definition of the Hessian of a G\^{a}teaux differentiable function and later prove an analogous formula to Proposition \ref{derprop} for the Hessian of an NC operator function. The formula is of similar flavor to one derived by Agler and \mc Carthy in \cite{amif16} for matricial nc functions.

  \begin{definition}
  Let $X$ and $Y$ be Banach spaces and $U\subset X$ be open. For a G\^{a}teaux differentiable function $g:U\rightarrow Y$, we define the \emph{Hessian} of $g$ at the point $x\in U$ to be \begin{align}\label{hessquo} Hg(x)[h,k]:=\lim_{t\rightarrow 0}\frac{Dg(x+tk)[h]-Dg(x)[h]}{t},\end{align} whenever the limit exists for all $h,k \in B(\ch)^d$.
  \end{definition}

  In the next lemma, we show that the derivative of an operator NC function is itself NC, that the Hessian exists for NC functions, and that the Hessian is again NC. As an application of these facts, we give a simple, calculus-based proof using boundedness of the Hessian that an operator NC function must, in particular, be of class $C^1.$ 

  \begin{lemma}\label{hessianexists}
Suppose $f:\om \subset B(\ch)^d \rightarrow B(\ch)^r$ is an NC function. 

\begin{enumerate}
\item The derivative map $\phi:\om \times B(\ch)^d \rightarrow B(\ch)^r$ given by $$\phi(x,h):=Df(x)[h]$$ is an NC function.

\item The Hessian $Hf(x)$ exists at all $x\in \om$ and the map $\om \times B(\ch)^{2d} \rightarrow B(\ch)^r$ given by $(x,h,k)\mapsto Hf(x)[h,k]$ is an NC function. Furthermore, $$Hf(x)[h,k]=D\phi(x,h)[k,0].$$ 

\item $f$ is $C^1.$
\end{enumerate}
  \end{lemma}

  \begin{proof}
(i) Let $\{\om_k\}$ be an exhaustion of $\om$ as in the definition of NC domain. A natural candidate for an NC exhausting sequence for $\om \times B(\ch)^d$ is $$W_k:=\om_k \times \{h\in B(\ch)^d : \Vert h\Vert \leq k\}.$$ Indeed, the requirements of Definition \ref{ncd} are readily seen, so $\om \times B(\ch)^d$ is an NC domain. 

We now show $\phi$ is an NC function. This is a simple matter of using the definition of the derivative. Let $(x_1,h_1)$ and $(x_2,h_2)$ be in $\om \times B(\ch)^d$ and let $s:\ch \rightarrow \ch^{(2)}$ be invertible such that $$(X,H):=s^{-1}\begin{bmatrix}
(x_1,h_1) & 0 \\
0 & (x_2,h_2)  
  \end{bmatrix}s \in \om \times B(\ch)^d.$$ Since $f$ is NC,\begin{align*} 
  \phi(X,H) &=  \lim_{t\rightarrow 0} \frac{1}{t}\left\{f\left(s^{-1}\begin{bmatrix}
x_1 & 0 \\
0 & x_2  
  \end{bmatrix}s +ts^{-1}\begin{bmatrix}
h_1 & 0 \\
0 & h_2  
  \end{bmatrix}s\right)-f\left(s^{-1}\begin{bmatrix}
x_1 & 0 \\
0 & x_2  
  \end{bmatrix}s\right)\right\} \\
  &= \lim_{t\rightarrow 0} s^{-1}\begin{bmatrix}
\frac{f(x_1+th_1)-f(x_1)}{t} & 0 \\
0 & \frac{f(x_2+th_2)-f(x_2)}{t} 
  \end{bmatrix}s \\
  &= s^{-1}\begin{bmatrix}
\phi(x_1,h_1) & 0 \\
0 & \phi(x_2,h_2)  
  \end{bmatrix}s,
  \end{align*} which proves part (i).

(ii) Since $\phi$ is NC on its domain, we apply Lemma \ref{difflemma} to conclude $\phi$ is G\^{a}teaux differentiable. Unraveling the definitions therefore shows that the Hessian $Hf(x)$ exists for all $x \in \om,$ and the equality $Hf(x)[h,k]=D\phi(x,h)[k,0]$ must hold. Applying the result in part (i) to the NC function $\phi$ shows the map $\om \times B(\ch)^{3d}\rightarrow B(\ch)^r$ given by $(x,h,k,k')\mapsto D\phi(x,h)[k,k']$ is NC. Therefore, the Hessian map $(x,h,k)\mapsto Hf(x)[h,k]=D\phi(x,h)[k,0]$ must also be NC on $\om \times B(\ch)^{2d}$.

(iii) By part (ii), it in particular holds that for every $x\in \om$, there is a norm ball $B$ about $x$ and $M>0$ such that $\Vert Hf(y)[h,k] \Vert \leq M \Vert h\Vert \Vert k\Vert$ for all $y \in B$ and all $h,k \in B(\ch)^d.$ 

Fix $x\in \om$. Choose a ball $B$ about $x$ and $M>0$ as above. Then for $y\in B$ and $h\in B(\ch)^d,$ the map $t \mapsto Hf(x+t(y-x))[h,y-x]$ is continuous on the interval $[0,1]$ by part (ii), so we may estimate  \begin{align*}
\Vert Df(y)[h]-Df(x)[h]\Vert &= \left\Vert \int_{0}^1 \frac{d}{dt} Df(x+t(y-x))[h]dt\right\Vert \\
&= \left\Vert \int_{0}^1 Hf(x+t(y-x))[h,y-x]dt\right\Vert \\
&\leq  \int_0^1 \Vert Hf(x+t(y-x))[h,y-x]\Vert dt \\
&\leq  M\Vert h\Vert \Vert y-x\Vert.
  \end{align*} By definition of the operator norm, it then holds that $$\Vert Df(y)-Df(x)\Vert \leq M\Vert y-x\Vert$$ for $y\in B.$
  \end{proof}
 
Finally, we have the aforementioned formula for the Hessian of operatorial NC functions:

  \begin{proposition}\label{hessprop}
  Let $f:\om \subset B(\ch)^d \rightarrow B(\ch)^r$ be NC. Suppose $x\in \om$ and $u,v:\ch\rightarrow \ch^{(2)}$ are unitaries. Then for all $h,k\in B(\ch)^d$ of sufficiently small norm, \begin{align}\begin{split} \label{hess} 
   &f\left(v^{-1}\begin{bmatrix}
u^{-1}\begin{bmatrix}
x & k \\
0 & x
\end{bmatrix}u & u^{-1}\begin{bmatrix}
h & 0 \\
0 & h
\end{bmatrix}u \\
0 & u^{-1}\begin{bmatrix}
x & k \\
0 & x
\end{bmatrix}u
  \end{bmatrix}v\right)  \\
  &= w^{-1}\begin{bmatrix}
f(x) & Df(x)[k] & Df(x)[h] & Hf(x)[h,k] \\
0 & f(x) & 0 & Df(x)[h] \\
0 & 0 & f(x) & Df(x)[k] \\
0 & 0 & 0 & f(x)
  \end{bmatrix}w,\end{split}\end{align} where we set $w:=(u\oplus u)v$.
  \end{proposition}

  \begin{proof}
For ease of reading, let us write \[X:=u^{-1}\begin{bmatrix}
x & k \\
0 & x
\end{bmatrix}u \hspace{20pt} {\rm and} \hspace{20pt} H:=u^{-1}\begin{bmatrix}
h & 0 \\
0 & h
\end{bmatrix}u.\]  By closure under direct sums and unitary invariance, $X\in \om$ for $\Vert k\Vert$ sufficiently small, and $$v^{-1}\begin{bmatrix}
X & H \\
0 & X
\end{bmatrix}v \in \om$$ for $\Vert h\Vert$ sufficiently small. We may then compute, by letting $\phi$ be the derivative as in Lemma \ref{hessianexists}, $$\phi(X,H)=u^{-1}\begin{bmatrix}
\phi(x,h) & D\phi(x,h)[k,0] \\
0 & \phi(x,h) 
\end{bmatrix}u=u^{-1}\begin{bmatrix}
Df(x)[h] & Hf(x)[h,k] \\
0 & Df(x)[h] 
\end{bmatrix}u.$$
The left-hand side of (\ref{hess}) is then equal to \begin{align*}
f&\left(v^{-1}\begin{bmatrix}
X & H \\
0 & X
\end{bmatrix}v\right) = v^{-1}\begin{bmatrix}
f(X) & Df(X)(H)\\
0 & f(X)
\end{bmatrix}v \\
&=v^{-1}\begin{bmatrix}
f(X) & u^{-1}\begin{bmatrix}
Df(x)[h] & Hf(x)[h,k] \\
0 & Df(x)[h] 
\end{bmatrix}u \\
0 & f(X)
\end{bmatrix}v \\
&= v^{-1}\begin{bmatrix}
u^{-1}\begin{bmatrix}
 f(x) & Df(x)[k] \\
 0 & f(x)
\end{bmatrix}u & u^{-1}\begin{bmatrix}
Df(x)[h] & Hf(x)[h,k] \\
0 & Df(x)[h] 
\end{bmatrix}u \\
0 & 
u^{-1}\begin{bmatrix}
 f(x) & Df(x)[k] \\
 0 & f(x)
\end{bmatrix}u
\end{bmatrix}v,
\end{align*} which is equal to the right-hand side of (\ref{hess}).
  \end{proof}

  We note that it is possible to derive similar, albeit increasingly complicated formulas for higher order derivatives of NC functions, but we will be content with doing so only for the first derivative and the Hessian, as this is sufficient for our purposes and it illustrates the general principles behind derivative formulas of NC functions on operatorial domains. 

\section{STRONG NC FUNCTIONS AND THE SHIFT FORM} \label{sec: strongshift}

As discussed briefly in the introduction and Section \ref{sec:main}, we want to impose additional requirements of SOT-closedness of each level in an exhaustion of an NC domain, and that of SOT continuity of NC operator functions in order to relax the instances of the hypothesis of the derivative satisfying the NC bounded below property to merely being \emph{injective} at all points. In practice, checking such boundedness below may be difficult in certain cases, but injectivity will typically be more readily verified. 

For $\eps>0$, we call the set $\{x\in B(\ch)^d : \text{dist}\,(x,U)<\eps\},$ of points in $B(\ch)^d$ with distance less than $\eps$ from the set $U,$ the $\eps$\emph{-neighborhood} of $U.$

\begin{definition}\label{sncd}
We say $\om \subset B(\ch)^d$ is a \emph{strong NC domain} if there exists an exhausting sequence $\{\om_k\}_{k=1}^{\infty}$ of $\om$ as in Definition \ref{ncd}, with the additional requirements that 

\begin{enumerate}
\item Each $\om_k$ is closed in the strong operator topology.

\item For each $k$ there is $\eps_k>0$ such that $\om_{k+1}$ contains the $\eps_k$-neighborhood of $\om_k.$
\end{enumerate}
\end{definition}

\begin{definition} \label{sncf}
Let $\om\subset B(\ch)^d$. A function $f:\om \rightarrow B(\ch)^r$ is called a \emph{strong NC function} if

\begin{enumerate}
\item There exists an exhausting sequence $\{\om_k\}_{k=1}^{\infty}$ of $\om$ as in Definition \ref{sncd} such that each restriction $f|_{\om_k}$  is continuous in the strong operator topology. 

\item $f$ is an NC function.
\end{enumerate}
\end{definition}

Since the strong operator topology is metrizable on norm-bounded subsets of $B(\ch)^d$ when $\ch$ is separable, the continuity condition (i) in Definition \ref{sncf} is equivalent to the following sequential criterion: for every $k,$  whenever $x_n$ is a sequence in $\om_k$ with $x_n\rightarrow x$ in SOT, we have $f(x_n)\rightarrow f(x)$ in SOT. Similarly, the condition of each $\om_k$ being SOT-closed in Definition \ref{sncd} is equivalent to a sequential characterization.

We remark further about Definitions \ref{sncd} and \ref{sncf}. Any $B_{\delta},$ as described in Section \ref{sec: prelim}, is a strong NC domain since the exhaustion given in (\ref{bdelta}) satisfies the additional requirements of Definition \ref{sncd}. Indeed, such a $\delta$ is Lipschitz on bounded sets and multiplication is strongly continuous on bounded sets. Moreover, as noncommutative polynomials and rational functions (such as the example in (\ref{example}) on the bidisk) are strongly continuous on appropriate norm-bounded sets, in practice these additional requirements seem rather mild and natural.

Secondly, condition (ii) in Definition \ref{sncd} is just a technical strengthening of the condition $\om_k \subset \emph{\inn}\om_{k+1}$ (which we have been using so far), and it ensures that the derivative of a strong NC function is also a strong NC function. Indeed, if $f:\om \subset B(\ch)^d \rightarrow B(\ch)^r$ is a strong NC function, say with exhausting sequence $\om_k$ as in Definition \ref{sncf}, taking the obvious exhaustion of $\om \times B(\ch)^d$ shows it is a strong NC domain. Furthermore, for every $k,$ whenever $x_n$ is a sequence in $\om_k$ with $x_n\rightarrow x$ in SOT and whenever $h_n\rightarrow h$ in SOT, we have $Df(x_n)[h_n]\rightarrow Df(x)[h]$ in SOT. To see this, fix $k$ and note that by closure under direct sums and unitary invariance of $\om_k,$ there is a unitary $u:\ch\rightarrow \ch^{(2)}$ such that $$u^{-1}\begin{bmatrix}
x_n & 0 \\
0 & x_n
\end{bmatrix}u \in \om_k$$ for all $n$. As the strongly convergent sequence $h_n$ is bounded, condition (ii) in Definition \ref{sncd} implies there is $\eps>0$ (independent of $n$) such that $$u^{-1}\begin{bmatrix}
x_n & \eps h_n \\
0 & x_n
\end{bmatrix}u \in \om_{k+1}$$ for all $n.$ Therefore, by Proposition \ref{derprop} and because $f|_{\om_{k+1}}$ is strongly continuous, 
\begin{align*}
 u^{-1}\begin{bmatrix}
f(x) & \eps Df(x)[h] \\
0 & f(x)
\end{bmatrix}u &= f\left(u^{-1}\begin{bmatrix}
x & \eps h \\
0 & x
\end{bmatrix}u\right) \\
&= \lim_{n\rightarrow \infty} f\left(u^{-1}\begin{bmatrix}
x_n & \eps h_n \\
0 & x_n
\end{bmatrix}u\right) \\
&= \lim_{n\rightarrow \infty} u^{-1}\begin{bmatrix}
f(x_n) & \eps Df(x_n)[h_n] \\
0 & f(x_n)
\end{bmatrix}u,\end{align*} where all limits are in the strong operator topology. Therefore, we conclude $Df(x_n)[h_n]\rightarrow Df(x)[h]$ in SOT.

In order to prove non-trivial results such as Theorem \ref{strongbb} for strong NC functions, we turn our attention to the notion of "shift forms". The following construction is motivated by the dilation theory introduced by A. Frazho \cite{FRAZHO2}, \cite{FRAZHO} and G. Popescu \cite{pop3}, \cite{pop4} and was privately communicated to the author by J. E. Pascoe. Similar ideas in a different setting were utilized in \cite{passerpascoe}. 

The separability of the underlying Hilbert space will now be used extensively. Throughout this section, we fix a countable orthonormal basis $\{e_1, e_2, \ldots\}$ for $\ch.$ Given a $d$-tuple $X\in B(\ch)^d,$ the idea is to find a unitary operator in $B(\ch)$ which provides a basis for $\ch$ on which the coordinates of $X$ essentially act as shifts.

Let $M$ be the shift operator $Me_k=e_{k+1}$. For the sake of brevity, we write $(X,M)$ for the $(d+1)$-tuple $(X^1,\ldots,X^d,M).$ 
We will denote the complex vector space of polynomials in $(d+1)$ noncommuting variables of degree less than or equal to $k$
 by $\mathcal{P}(k,d)$ and write $\alpha(k,d)$ for its dimension.
Begin by defining a nested sequence of subspaces of $\ch:$
\[V_k^X:= \{p(X,M)e_1 : p \in  \mathcal{P}(k,d)\},
\] for $k\geq 0.$ We record the following properties of the $V_k^X:$ 

\begin{enumerate}
\item For each $k\geq 0,$ we have $e_1,\ldots, e_{k+1}\in V_k^X$. In particular, $$\ch=\overline{\bigcup_{k=0}^{\infty} V_k^X}.$$

\item The inclusion $X^i V_k^X \subset V_{k+1}^X$ holds for all $i=1,\ldots,d$ and $k\geq 0.$ 

\item The $V_k^X$ form a strictly increasing sequence.

\item The inequality $$\dim V_k^X \leq \alpha(k,d)$$ holds for all $k\geq 0,$ independent of the choice of $d$-tuple $X.$  
\end{enumerate}

Properties (i), (iii), and (iv) above imply that there exists a unitary operator $u\in B(\ch)$, depending on $d$ and $X,$ but not $k$, such that \begin{align} \label{sh1}
u\,(\text{span}\, \{e_1,\ldots,e_k, e_{k+1}\})\subset V_k^X
\end{align}
and 
\begin{align} \label{sh2}
u^*(V_k^X)\subset \text{span}\, \{e_1,\ldots, e_{\alpha(k,d)}\}
\end{align}
hold for every $k\geq 0.$ For a unitary $u$ satisfying (\ref{sh1}) and (\ref{sh2}), we call the $d$-tuple \begin{align*}\widetilde{X}:=u^*Xu\end{align*} a \emph{shift form} of $X.$ We note that there may well be more than one such unitary for a given $X,$ but for our purposes, the existence of at least one is sufficient. Moreover, the results proved in the present section are independent of choice of shift form; all that is required are the four properties listed above. 

This construction allows us to prove an SOT compactness-like theorem for bounded subsets of $B(\ch)^d$. It is well-known that the unit ball of $B(\ch)$ is not SOT (sequentially) compact when $\ch$ is infinite dimensional, but we prove in Lemma \ref{sotcompact} that for any bounded sequence in $B(\ch)^d,$ there is a subsequence along which its sequence of shift forms converge SOT. More precisely, in fact, given a bounded sequence $X_n \in B(\ch)^d,$ and given any sequence of unitaries $u_n$ such that $u_n^*X_nu_n$ is a shift form of $X_n$ for each $n,$ there is a subsequence along which $u_n^*X_nu_n$ converges in SOT. This statement lends itself nicely to applications with strong NC functions since they preserve conjugations by unitary operators and are strongly continuous when restricted to certain unitarily invariant sets. Moreover, we have sufficient norm control over the shift forms so that, after conjugating by further unitaries if necessary, the subsequential limit will have large norm if the original sequence is bounded away from zero (see part (ii) of Lemma \ref{sotcompact}). 

Lemma \ref{sizeshift} below is a technical ingredient used in this note only in the proof of part (ii) of Lemma \ref{sotcompact} but is an interesting property of shift forms in its own right.

\begin{lemma}\label{sizeshift}
If $X\in B(\ch)^d$ and $k\geq 1$, then by letting $P_k$ denote the projection onto the subspace spanned by the first $k$ basis vectors $e_1,\ldots, e_k,$ we have \begin{align} \label{shiftinequality}
\Vert P_k X^iP_k\Vert \leq \Vert P_{\alpha(k,d)} \widetilde{X}^iP_{\alpha(k,d)} \Vert
\end{align} for each $i=1,\ldots, d$ and choice of shift form $\widetilde{X}$ of $X.$
\end{lemma}

\begin{remark}\rm
The proof shows, in fact, that the norm inequality (\ref{shiftinequality}) can be refined slightly. For example, under the hypotheses of Lemma \ref{sizeshift}, it holds that \begin{align*} \Vert X^iP_k\Vert \leq \Vert P_{\alpha(k,d)} \widetilde{X}^iP_{\alpha(k-1,d)} \Vert.
\end{align*} Since we do not require this inequality moving forward, we opt for the more visually symmetric (\ref{shiftinequality}).
\end{remark}

\begin{proof}
Write $\widetilde{X}=u^*Xu$ for a unitary $u$ satisfying (\ref{sh1}) and (\ref{sh2}). Let $y \in \text{span}\,\{e_1,\ldots, e_k\}$ with $\Vert y\Vert \leq1.$ Since $\text{span}\,\{e_1,\ldots, e_k\} \subset V_{k-1}^X,$ we know by (\ref{sh2}) that the containment $u^*y \in \text{span}\, \{e_1,\ldots, e_{\alpha(k-1,d)}\}$ holds. Furthermore, this implies $X^iy\in V_k^X$, and so $u^*X^iy \in \text{span}\, \{e_1,\ldots, e_{\alpha(k,d)}\}.$ Therefore we may estimate  \begin{align*} \Vert P_kX^iP_k y\Vert &\leq \Vert X^i y\Vert = \Vert u^*X^iy\Vert \\
&=\Vert P_{\alpha(k,d)}u^*X^iy\Vert =\Vert P_{\alpha(k,d)}[u^*X^iu]u^*y\Vert \\
 &= \Vert P_{\alpha(k,d)}[u^*X^iu]P_{\alpha(k,d)}u^*y\Vert \\
 &\leq \Vert P_{\alpha(k,d)} \widetilde{X}^iP_{\alpha(k,d)} \Vert.\end{align*} Taking supremum over such $y$ finishes the proof.
\end{proof}

Part (ii) of Lemma \ref{sotcompact} will be used in the proof of Theorem \ref{strongbb}. In the notation of this lemma, we need $H'\neq 0$ to ensure it is not in the kernel of any injective derivative map of a strong NC function.

\begin{lemma}\label{sotcompact}
The following two convergence properties hold.

\begin{enumerate}
\item Let $X_n$ be a bounded sequence in $B(\ch)^d.$ For any sequence of shift forms $\widetilde{X_n}$ of $X_n$, there is a subsequence along which $\widetilde{X_n}$ converges in SOT. In particular, given a bounded sequence $X_n$ in $B(\ch)^d$, there exists a sequence of unitaries $U_n$ such that $U_n^*X_nU_n$ converges in SOT along a subsequence.

\item Suppose $X \in B(\ch)^d$ and $H_n\in B(\ch)^d$ with $\Vert H_n\Vert =1$ for all $n.$ Then there exist unitaries $W_n \in B(\ch),$ a point $(X',H')\in B(\ch)^{2d}$ with $H'\neq 0$, and a subsequence along which  $$W_n^*(X,H_n)W_n \rightarrow (X',H')$$ in SOT.
\end{enumerate}
\end{lemma}

\begin{proof}
(i) For each $n$, let $\widetilde{X_n}=u_n^*X_nu_n$ be any shift form of $X_n$. For every $n$, $i=1,\ldots, d$, and $k\geq 1,$ properties (\ref{sh1}) and (\ref{sh2}) imply \begin{align*}
\widetilde{X_n}^i (\text{span}\, \{e_1,\ldots,e_k\}) &= u_n^*X_n^iu_n (\text{span}\, \{e_1,\ldots,e_k\}) \\
&\subset u_n^*X_n^i (V_{k-1}^{X_n}) \\
&\subset u_n^* (V_k^{X_n})\\
&\subset \text{span}\, \{e_1,\ldots,e_{\alpha(k,d)}\}.
\end{align*} Therefore, for every $i=1,\ldots, d$ and $k\geq 1,$ the sequence $\{\widetilde{X_n}^i e_k\}_{n=1}^{\infty}$ is bounded and contained in a finite dimensional subspace. By a diagonalization argument, we may then find a subsequence $n_j$ so that $\widetilde{X_{n_j}}^i e_k$ converges for every $i=1,\ldots, d$ and $k\geq 1.$ By boundedness again, this implies $\widetilde{X_{n_j}}^i$ converges SOT for every $i=1,\ldots, d.$

(ii) By passing to a subsequence if necessary, we may assume there is $i\in \{1,\ldots, d\}$ such that $\Vert H_n^i \Vert=1$ for all $n.$ We again denote by $P_k$ the projection onto the subspace spanned by the first $k$ basis vectors $e_1,\ldots, e_k.$ First note that if $T \in B(\ch)$ has operator norm equal to 1, then for every $\eps>0$ small there exists a unitary $W \in B(\ch)$ such that $1-\eps \leq \Vert P_2 W^*TWP_2\Vert.$ This can be seen by choosing a unit vector $v$ which approximates the norm of $T$, and then defining a unitary which maps $e_1$ to $v$, and $e_2$ to a suitable linear combination $av+bTv$.

Applying this to $H_n^i$ for each $n,$ we can find unitaries $Q_n \in B(\ch)$ so that $$1-\frac{1}{n} \leq \Vert P_2 Q_n^*H_n^iQ_nP_2\Vert.$$  Since $X_n:=Q_n^*(X,H_n)Q_n$ is a bounded sequence in $B(\ch)^{2d},$ by part (i) there is a sequence of unitaries $U_n$ and a subsequence $n_j$ along which $\widetilde{X_n}=U_n^*X_nU_n$ converges in SOT, say to $(X',H') \in B(\ch)^{2d}$. Define $W_n:=Q_nU_n$. Combining this and (\ref{shiftinequality}) with $k=2$ and $2d$ variables, we estimate \begin{align*}
1-\frac{1}{j} &\leq  \Vert P_2 Q_{n_j}^*H_{n_j}^iQ_{n_j} P_2\Vert \\
&= \Vert P_2 X_{n_j}^{d+i} P_2\Vert \\
&\leq  \Vert P_{\alpha(2,2d)} \widetilde{X_{n_j}}^{d+i} P_{\alpha(2,2d)} \Vert \\
&= \Vert P_{\alpha(2,2d)} W_{n_j}^*H_{n_j}^iW_{n_j} P_{\alpha(2,2d)} \Vert . 
\end{align*} Since strong convergence implies norm convergence on finite dimensional spaces, taking the limit as $j\rightarrow \infty$ in the above estimate implies $$1\leq \Vert P_{\alpha(2,2d)} (H')^i P_{\alpha(2,2d)} \Vert \leq \Vert H'\Vert.$$ Therefore $H'\neq 0,$ which concludes the proof.
\end{proof}

\begin{remark}\rm
The proof of part (i) of Lemma \ref{sotcompact} shows further that we can choose a subsequence along which $u_n^*X_nu_n$ and $u_n^*$ both converge in SOT. This is often a useful property of the unitaries defining the shift forms since the SOT limit of $u_n^*$ is necessarily an isometry. These observations can be helpful in the study of convexity in the operator setting.
\end{remark}

\section{PROOFS OF MAIN RESULTS} \label{sec: proofs}

We now provide detailed proofs of the main results listed in Section \ref{sec:main}. To begin, we need a general result on linear maps between Banach spaces. As the author could not find a suitable source in the literature, we include its statement and simple proof below for convenience. $B(X,Y)$ denotes the bounded linear maps between the Banach spaces $X$ and $Y$, with the operator norm.

 \begin{lemma}\label{banach}
Let $X$ and $Y$ be  Banach spaces and fix $\alpha>0$. The set of maps in $B(X,Y)$ that are surjective and bounded below by $\alpha$ is norm-closed.
\end{lemma}

\begin{proof}
Let $T_n$ be a sequence of such linear maps converging to $T$. As $\alpha\Vert x\Vert \leq \Vert T_nx\Vert$ holds for all $n$ and $x\in X,$ we see that $\alpha\Vert x\Vert \leq \Vert Tx\Vert$ for all $x\in X,$ so $T$ is bounded below by $\alpha.$ 

Now we show $T$ must also be surjective. The uniform bound below on the $T_n$ implies the sequence of inverses $T_n^{-1}$ is uniformly bounded in operator norm by $1/\alpha.$ Thus, the estimate $$\Vert T_n^{-1}-T_m^{-1}\Vert \leq \Vert T_m^{-1}\Vert \Vert T_m-T_n\Vert \Vert T_n^{-1}\Vert \leq 1/\alpha^2 \Vert T_m-T_n\Vert$$ shows $T_n^{-1}$ is a convergent sequence in $B(Y,X)$. Since $T_nT_n^{-1}=1_Y$ for all $n,$ we immediately see that $T$ is surjective.
\end{proof}

We reiterate that the notion of connectedness in what follows is with respect to the norm topology on $B(\ch)^d$. It is suggested that the reader recall Definition \ref{ncbbp} of the NC bounded below property.

\begin{proof}[Proof of Theorem \ref{bbthm}]
By hypothesis, the set \[U:=\{x\in \om : Df(x) \,\,\text{is invertible}\}\] is non-empty. Since invertible maps form a norm-open set in $B(B(\ch)^d),$ the continuity of the map $x\mapsto Df(x)$ implies $U$ is open in norm.

As $\om$ is connected, it suffices to show $U$ is also closed in $\om.$ To that end, take a sequence $x_n$ in $U$ converging to  $x \in \om$. We claim there is a uniform $\alpha>0$ such that each $Df(x_n)$ is bounded below by $\alpha.$ To see this, take an exhaustion $\{\om_k\}$ of $\om$ as in Definition \ref{ncd}. Since $x_n\rightarrow x \in \om,$ and since the exhaustion satisfies $\om_k \subset \text{int}\,\om_{k+1}$, there is $k$ large enough so that all the $x_n$ lie in $\om_k.$ Since $\om_k$ is closed under countably infinite direct sums, there is $u:\ch \rightarrow \ch^{(\infty)}$ unitary such that $$z:=u^{-1}\begin{bmatrix}
x_1 & & \\
& x_2 & \\
& & \ddots
\end{bmatrix}u \in \om_k.$$ By the hypothesis of $Df$ satisfying the NC bounded below property, $Df(z)$ is bounded below, say by $\alpha>0.$ Now fix $n$ and let $h \in B(\ch)^d$ be arbitrary. Let $h_n$ denote the diagonal matrix with $h$ in the $n$th diagonal entry and 0 else. Since \begin{align} \label{derivdirect} Df(z)[u^{-1}h_nu]=u^{-1}\begin{bmatrix}
0 & & & & \\
& \ddots & & & \\
& & Df(x_n)[h] & & \\
& & & 0 & \\
& & & & \ddots
\end{bmatrix}u 
\end{align} holds by Lemma \ref{hessianexists} (i), we may take norms in (\ref{derivdirect}) to get $$\Vert Df(x_n)[h]\Vert = \Vert Df(z)[u^{-1}h_nu] \Vert \geq \alpha \Vert u^{-1}h_nu\Vert= \alpha \Vert h\Vert.$$ This implies each $Df(x_n)$ is bounded below by $\alpha.$  Again, since $f$ is $C^1$, we have $Df(x_n)\rightarrow Df(x)$ in norm so Lemma \ref{banach} implies $Df(x)$ is invertible. Thus $x\in U$ and $U$ is closed in $\om$.
\end{proof}

With this sufficient condition for global invertibility of the derivative of an NC function now obtained, we can prove our inverse function theorem:

\begin{proof}[Proof of Theorem \ref{inverse}]
By Theorem \ref{bbthm}, $Df(x)$ is an invertible linear mapping $B(\ch)^d\rightarrow B(\ch)^d$ for every $x\in \om.$ Theorem \ref{injthm} tells us that $f$ is then injective on $\om,$ so $f^{-1}$ exists as a map $f(\om)\rightarrow \om.$
We must show $f(\om)$ and $f^{-1}$ are both NC. In fact, we claim that if we take an exhausting sequence $\{\om_k\}$ for $\om$, then the sequence of images $\{f(\om_k)\}$ is an exhaustion for $f(\om)$.

First, we show $f(\om)$ is an NC domain. All required properties in Definition \ref{ncd} of the sequence $f(\om_k)$ are immediate from the corresponding properties of $\om_k$ and the fact that $f$ is NC, except possibly the containment $f(\om_k)\subset \text{int}\, f(\om_{k+1}).$ But since $f$ is $C^1,$  the classical inverse function theorem for Banach spaces (see \cite{banachbook} for a reference) implies $f$ is an open map because each $Df(x)$ is invertible. Hence, $$f(\om_k)\subset f(\text{int}\, \om_{k+1}) =\text{int}\, f(\text{int}\, \om_{k+1}) \subset \text{int}\, f(\om_{k+1}).$$

Finally, we show $f^{-1}$ is an NC function. Let $f(x_1)$ and $f(x_2)$ be in $f(\om)$ and let $s:\ch\rightarrow \ch^{(2)}$ be invertible with \begin{align} \label{btrick} s^{-1}
\begin{bmatrix}
f(x_1) & 0 \\
0 & f(x_2) 
\end{bmatrix}s \in f(\om).\end{align} It suffices to show $w:=s^{-1}\begin{bmatrix}
x_1 & 0 \\
0 & x_2 
\end{bmatrix}s$ lies in $\om,$ since we may then apply $f$ and use the fact that $f$ preserves direct sums to get $$f^{-1}\left(s^{-1}
\begin{bmatrix}
f(x_1) & 0 \\
0 & f(x_2) 
\end{bmatrix}s\right)=s^{-1}
\begin{bmatrix}
x_1 & 0 \\
0 & x_2
\end{bmatrix}s.$$ This then shows $f^{-1}$ preserves direct sums. Note that the membership $w\in \om$ does not immediately follow since $s$ is not necessarily unitary. To that end, call the expression in (\ref{btrick}) $f(z)$ for a unique $z \in \om.$ We know there is a unitary $u:\ch \rightarrow \ch^{(2)}$ such that $$x:=u^{-1}\begin{bmatrix}
x_1 & 0 \\
0 & x_2 
\end{bmatrix}u \in \om.$$ Since $f$ is NC, if we define $L:=s^{-1}u \in B(\ch),$ then $$f(x)=u^{-1}\begin{bmatrix}
f(x_1) & 0 \\
0 & f(x_2) 
\end{bmatrix}u = L^{-1}f(z)L.$$ We claim that $z=LxL^{-1},$ which proves $w \in \om,$ since $LxL^{-1}=w.$ There is unitary $v:\ch\rightarrow \ch^{(2)}$ such that $v^{-1}\begin{bmatrix}
z & 0 \\
0 & x 
\end{bmatrix}v \in \om$. For sufficiently small  $\eps>0,$ apply Lemma \ref{L}: \begin{align*}f\left(v^{-1}\begin{bmatrix}
  z & \eps(Lx-zL) \\
  0 & x
  \end{bmatrix}v\right)&=v^{-1}\begin{bmatrix}
  f(z) & \eps(Lf(x)-f(z)L) \\
  0 & f(x)
  \end{bmatrix}v \\
  &= v^{-1}\begin{bmatrix}
  f(z) & 0 \\
  0 & f(x)
  \end{bmatrix}v \\
  &= f\left(v^{-1}\begin{bmatrix}
  z & 0 \\
  0 & x
  \end{bmatrix}v\right).\end{align*} It now follows from injectivity of $f,$ that $Lx=zL,$ as desired.
\end{proof}

Recall from Section \ref{sec:main} the notation $Z_f$ denotes the zero set of the function $f$. We now prove the implicit function theorem for NC operator functions by using Theorem \ref{inverse} applied to an appropriate auxiliary function. The derivative map of this function will be shown to also have the NC bounded below property under the hypotheses of Theorem \ref{implicit}.

\begin{proof}[Proof of Theorem \ref{implicit}]
Consider the NC function $F:\om \rightarrow B(\ch)^d$ given by the formula $F(x)=(x^1,\ldots, x^{d-r},f(x)).$ We claim that $F$ satisfies the hypotheses of Theorem \ref{inverse}. The derivative of $F$ is computed as \begin{align}\label{implicitdiff}
DF(x)[h]=(h^1,\ldots,h^{d-r},Df(x)[h]). 
\end{align} 

We first show that $DF(a)$ is invertible in $B(B(\ch)^d)$, where $a\in \om$ is the point such that $\Psi(a)$ is assumed to be invertible. Let $(v,w)\in B(\ch)^{d-r}\times B(\ch)^r$ be arbitrary. By hypothesis, there is $(h^{d-r+1},\ldots, h^d)\in B(\ch)^r$ such that $$Df(a)[0,\ldots,0,h^{d-r+1},\ldots, h^d]=w-Df(a)[v,0,\ldots,0].$$ Linearity of the derivative and (\ref{implicitdiff}) then give \begin{align*}
DF(a)[v,h^{d-r+1},\ldots, h^d] &= (v,Df(a)[v,h^{d-r+1},\ldots, h^d]) \\
&=(v,Df(a)[v,0,\ldots,0]+Df(a)[0,\ldots,0,h^{d-r+1},\ldots, h^d]) \\
&= (v,w),
\end{align*} so $DF(a)$ is surjective. As $DF(a)$ is clearly injective when $\Psi(a)$ is, we conclude that $DF(a)$ is invertible.

We now show that $DF$ has the NC bounded below property by showing, for $x\in \om$, that $DF(x)$ is bounded below if and only if $\Psi(x)$ is bounded below. It is immediate to see that $\Psi(x)$ is bounded below if $DF(x)$ is, so we prove the converse. Fix $x\in \om$ such that $\Psi(x)$ is bounded below. Then there is $\eps>0,$ depending only on $x,$ such that $$\Vert Df(x)[0,\ldots,0,h^{d-r+1},\ldots, h^d]\Vert \geq \eps \max \{\Vert h^{d-r+1}\Vert, \ldots, \Vert h^d\Vert\}$$ for all $(h^{d-r+1},\ldots, h^d)\in B(\ch)^r.$ Therefore we may estimate \begin{align*}
\Vert Df(x)[h^1,\ldots,h^d]\Vert &= \Vert Df(x)[h^1,\ldots,h^{d-r},0,\ldots,0] \\
&\hspace{90pt}+Df(x)[0,\ldots,0,h^{d-r+1},\ldots,h^d]\Vert \\
&\geq \Vert Df(x)[0,\ldots,0,h^{d-r+1},\ldots,h^d]\Vert \\
&\hspace{90pt}-\Vert Df(x)[h^1,\ldots,h^{d-r},0,\ldots,0]\Vert \\
&\geq \eps \max\{\Vert h^{d-r+1}\Vert, \ldots, \Vert h^d\Vert\} \\
&\hspace{90pt}-\Vert Df(x)\Vert \max \{\Vert h^1\Vert,\ldots,\Vert h^{d-r}\Vert\}.
\end{align*} This combined with taking norms in (\ref{implicitdiff}) gives us \begin{align*}\Vert DF(x)[h^1,\ldots, h^d]\Vert&=\max \{\Vert h^1\Vert,\ldots,\Vert h^{d-r}\Vert,\Vert Df(x)[h^1,\ldots,h^d]\Vert\} \\
&\geq \frac{\eps}{\eps+\Vert Df(x)\Vert +1} \max \{\Vert h^1\Vert,\ldots,\Vert h^d\Vert\},
\end{align*}  so $DF(x)$ is bounded below.

Therefore, by Theorem \ref{inverse}, we know $F^{-1}: F(\om)\rightarrow \om$ is NC. We may write $F^{-1}$ it terms of its coordinates, say $F^{-1}=(G^1,\ldots, G^d).$ Let $V$ be the projection onto the first $d-r$ coordinates of the zero set $Z_f.$ Thus, $V$ can explicitly be written as the set of $y\in B(\ch)^{d-r}$ such that there exists $z\in B(\ch)^r$ with $(y,z)\in \om$ and $f(y,z)=0.$ Then $V$ is seen to be an NC domain with exhaustion $\{V_k\}$, where $V_k$ is defined to be the set of $y\in B(\ch)^{d-r}$ such that there exists $z\in B(\ch)^r$ with $(y,z)\in \om_k$ and $f(y,z)=0.$ (The containment $V_k\subset \text{int} \, V_{k+1}$ follows since $F$ is an open map by Theorem \ref{bbthm} and the classical Banach space inverse function theorem.) Now define $\phi:V\rightarrow B(\ch)^r$ by $$\phi(y):=(G^{d-r+1}(y,0),\ldots, G^d(y,0)).$$ It is immediate to check that $\phi$ is an NC function.

Let $y\in V.$ From the definitions, \begin{align*}
(y,0)&= F(G^1(y,0),\ldots, G^{d-r}(y,0),\phi(y)) \\
&= (G^1(y,0),\ldots, G^{d-r}(y,0),f(F^{-1}(y,0))).
\end{align*} Therefore $y=(G^1(y,0),\ldots, G^{d-r}(y,0))$ and $f(F^{-1}(y,0))=0$, so $(y,\phi(y))\in Z_f.$ Conversely, let $x=(y,z)\in Z_f,$ where $y \in B(\ch)^{d-r}$ and $z\in B(\ch)^r$. Then $y\in V$ and $F(x)=(y,f(x))=(y,0).$ Thus, $(y,z)=F^{-1}(y,0)$, which implies $z=\phi(y).$ This establishes the desired parametrization of $Z_f.$
\end{proof}

Now we come to the proofs of the main results concerning \emph{strong} NC functions. Theorem \ref{strongbb} and its corollary are our primary applications of the shift form construction from Section \ref{sec: strongshift}. The reader may want to review that section before proceeding with the following proof.

\begin{proof}[Proof of Theorem \ref{strongbb}]
Suppose there is $x \in \om$ such that $Df(x)$ is not bounded below. Then we can find a sequence $h_n\in B(\ch)^d$ of unit vectors such that \[\Vert Df(x)[h_n]\Vert \rightarrow 0.\] Let $\{\om_k\}$ be an exhaustion for $\om$ as in Definition \ref{sncf} and say the point $x$ lies in $\om_k.$ By Lemma \ref{sotcompact} (ii), there are unitaries $v_n$ and a point $(x',h')\in B(\ch)^{2d}$ with $h'\neq 0$ such that $$v_n^*(x,h_n)v_n\rightarrow (x',h')$$ in SOT along a subsequence $n_j$. Since $\om_k$ is unitarily invariant and SOT-closed, it follows that $v_{n_j}^*xv_{n_j} \in \om_k$ for every $j$ and that $x'\in \om_k.$ Therefore, by the discussion following Definition \ref{sncf} on the SOT continuity of the derivative of a strong NC function, we have $$v_{n_j}^*Df(x)[h_{n_j}]v_{n_j}=Df(v_{n_j}^*xv_{n_j})[v_{n_j}^*h_{n_j}v_{n_j}] \rightarrow Df(x')[h']$$ in SOT. But by the choice of $h_n$ and because the $v_n$ are unitary, we also have $$v_{n_j}^*Df(x)[h_{n_j}]v_{n_j}\rightarrow 0$$ in norm. Therefore, $Df(x')[h']=0,$ contradicting the hypothesis of injectivity of $Df(x').$
\end{proof}

\begin{proof} [Proof of Corollary \ref{strongbbcor}]
Apply Theorem \ref{injthm} to conclude that each $Df(x)$ is injective. Then by Theorem \ref{strongbb}, each $Df(x)$ is in fact bounded below since $f$ is strong NC. Theorem \ref{bbthm} then implies the desired conclusion.
\end{proof}

\bigskip

 \emph{Acknowledgements.} This work was partially supported by the National Science Foundation Grant DMS 1565243.

\bibliography{Inverse} 
\bibliographystyle{acm}

\address

\end{document}